\documentclass[a4paper,12pt]{article}

\usepackage[top=3.0cm,bottom=3.0cm,left=2.25cm,right=2.25cm]{geometry}
\usepackage{amsmath,amsthm,mathrsfs,graphicx,amsfonts}
\usepackage{bm}
\usepackage{cases}
\usepackage[english]{babel}
\usepackage{amsmath,amsthm}
\usepackage{amsfonts}
\usepackage{latexsym}
\usepackage{graphicx}
\usepackage{txfonts}
\usepackage[numbers,sort&compress]{natbib}
\usepackage[natural]{xcolor}
\usepackage{rotating}
\usepackage{mathtools}
\usepackage{enumitem}

\newtheorem{lem}{Lemma}[section]
\newtheorem{thm}[lem]{Theorem}
\newtheorem{cor}[lem]{Corollary}

\newtheorem{claim}[lem]{\indent Claim}
\newtheorem{conj}[lem]{Conjecture}

\begin{document}
\title{Degree sequences realizing labelled \(h\)-factors}
\date{}
\author{Zhen Liu\footnote{Email: 1552580575@qq.com}, ~~Qinghou Zeng\footnote{Research supported by National Key R\&D Program of China (Grant No. 2023YFA1010202) and National Natural Science Foundation of China (Grant No. 12371342). Email: zengqh@fzu.edu.cn (Corresponding
author)}\\
{\small Center for Discrete Mathematics, Fuzhou University, Fujian, 350003, China}}

\maketitle

\maketitle
\begin{abstract}
For a positive integer \( k \), let \( [k] = \{1, 2, \ldots, k\} \). Let \( h \) be a non-negative integer, and let \( n \) be a multiple of \( h + 1 \). Define \( H \) as the disjoint union of \( n/(h+1) \) cliques (each of size \( h + 1 \)) with vertex sets \( V_1, \ldots, V_{n/(h+1)} \), where \( V_i = \{ v_j \mid j = (i-1)(h+1) + k, k \in [h+1] \} \) for \( i \in [n/(h+1)] \). A non-increasing integer sequence \( (d_1, \ldots, d_n) \) is \( H \)-realizable if there exists a graph \( G \) with \( V(G) = V(H) = \{ v_i \mid i \in [n] \} \), \( d_G(v_i) = d_i \) for all \( i\in [n] \), and \( G \) contains \( H \) as a spanning subgraph. If \( h = 0 \), then a non-increasing integer sequence \( (d_1, \ldots, d_n) \) is \( H \)-realizable if and only if there exists a graph \( G \) with degree sequence \( (d_1, d_2, \dots, d_n) \); Erd\H{o}s and Gallai established a necessary and sufficient condition for this property. Recently, Briggs, McDonald, and Shan extended their result to the case \( h = 1 \). In this paper, we establish a necessary and sufficient condition for a sequence \( (d_1, d_2, \dots, d_n) \) to be \( H \)-realizable for any non-negative integer \( h \), thereby confirming a conjecture due to Briggs, McDonald and Shan.
\end{abstract}

\section{Introduction}\label{Intro}

All graphs considered in this paper are simple. Let \( n \in \mathbb{N} \) and let \( d_1 \geq d_2 \geq \dots \geq d_n \) be non-negative integers. Say a sequence \( \pi = (d_1, d_2, \dots, d_n) \) is \emph{graphic} if there exists a graph \( G \) with degree sequence \( (d_1, d_2, \dots, d_n) \), and we call such a graph \( G \) a \emph{realization} of \( \pi \). For a graph $H$, a non-increasing non-negative integer sequence \(\pi= (d_1, d_2, \dots, d_n) \) is said to be \( H \)-\emph{realizable} if there exists a graph \( G \) such that \( V(G) = V(H) = \{ v_i \mid i \in [n] \} \), \( d_G(v_i) = d_i \) for all \( i \in [n] \), and \( G \) contains \( H \) as a spanning subgraph.  

For a positive integer \( k \), let \( [k] = \{1, 2, \dots, k\} \). Let \( h \) be a non-negative integer, and let \( n \) be a multiple of \( h + 1 \). Define \( H_h \) as the disjoint union of \( \frac{n}{h+1} \) cliques (each of size \( h + 1 \)) with vertex sets \( V_1, \dots, V_{\frac{n}{h+1}} \), where \( V_i = \{ v_j \mid j = (i-1)(h+1) + k, k \in [h+1] \} \) for \( i \in [\frac{n}{h+1}] \).
If \( h = 0 \), a non-increasing non-negative integer sequence \( \pi=(d_1, d_2, \dots, d_n) \) is \( H_0 \)-realizable if and only if it is graphic. Erd\H{o}s and Gallai \cite{erdos1960} established a classic result stated as follows.
\begin{thm}[Erd\H{o}s and Gallai \cite{erdos1960}]\label{thm:erdos_gallai}
Let \( n \in \mathbb{N} \) and let \( d_1 \geq d_2 \geq \dots \geq d_n \) be non-negative integers. The sequence \( \pi=(d_1, d_2, \dots, d_n) \) is graphic if and only if \( \sum_{i=1}^n d_i \) is even, and for every \( k \in [n] \),
\begin{align}\label{eq:erdos_gallai_ineq}
\sum_{i=1}^k d_i \leq k(k-1) + \sum_{i=k+1}^n \min\{d_i, k\}.
\end{align}
\end{thm}

Recently, Briggs, McDonald, and Shan \cite{6} addressed the problem of characterizing \( H_1 \)-realizable non-increasing non-negative integer sequences and obtained the following result.
\begin{thm}[Briggs, McDonald, and Shan \cite{6}]\label{thm:h1_realizable}
Let \( n \in \mathbb{N} \) and \( d_1 \geq d_2 \geq \dots \geq d_n \ge1\) be integers. The sequence \( \pi=(d_1, d_2, \dots, d_n) \) is \( H_1 \)-realizable if and only if \( \sum_{i=1}^n d_i \) is even, \( n \) is even, and for every \( k \in [n] \),
\[
\sum_{i=1}^k d_i \leq
\begin{cases}
k(k-1) + \sum_{i=k+1}^n \min\{d_i - 1, k\} & \text{if } k \text{ is even}, \\
k(k-1) + \min\{d_{k+1}, k\} + \sum_{i=k+2}^n \min\{d_i - 1, k\} & \text{if } k \text{ is odd}.
\end{cases}
\]
\end{thm}

Building on this work, Briggs, McDonald, and Shan \cite{6} further proposed a conjecture regarding the characterization of \( H \)-realizable sequences for general positive integers \( h \), which is formulated as follows.
\begin{conj}[Briggs, McDonald, and Shan \cite{6}] \label{1111}
Let \( n \in \mathbb{N} \) and \( d_1 \geq d_2 \geq \dots \geq d_n \ge h\) be positive integers. The sequence \( \pi=(d_1, d_2, \dots, d_n) \) is \( H_h \)-realizable if and only if \( \sum_{i=1}^n d_i \) is even, \( n \) is a multiple of \( h+1 \), and for every \( k \in [n] \),
\begin{align}\label{eq:h_realizable_main_ineq}
\sum_{i=1}^k d_i \leq k(k-1) + \sum_{i=k+1}^{k+1+h-s} \min\{d_i - h + s, k\} + \sum_{i=k+1+h-s+1}^n \min\{d_i - h, k\},
\end{align}
where \( s \in \{0, \dots, h\} \) and \( s \equiv k \pmod{h + 1} \).
\end{conj}

In this paper, we confirm Conjecture \ref{1111}, and we mention that our proof is inspired by the proof of Theorem \ref{thm:erdos_gallai} due to Tripathi, Venugopalan and West \cite{west}. We also have the following result as a corollary.
\begin{cor}\label{Corol}
Let \(h,n\in\mathbb{N}\) and let \(d_1\ge d_2\ge\cdots\ge d_n\ge h\) be positive integers. The sequence \( (d_1, d_2, \dots, d_n) \) can realize \(h\) disjoint perfect matchings in the same graph when \(\sum_{i=1}^{n}d_i\) is even, \(n\) is a multiple of \(h+1\), and \eqref{eq:h_realizable_main_ineq} holds for every \(k\in[n]\).
\end{cor}

In fact, Corollary \ref{Corol} is established by Briggs, McDonald, and Shan \cite{6} under the assumption that Conjecture \ref{1111} holds. 
Brualdi \cite{brualdi} and independently Busch, Ferrera, Hartke, Jacobsen, Kaul, and West \cite{2012west} conjectured that for a degree sequence \( \pi = (d_1, \dots, d_n) \) with an even \( n \), \( \pi \) admits a realization that contains \( h \) pairwise disjoint perfect matchings if and only if \( (d_1 - h, \dots, d_n - h) \) is graphic. While this conjecture has not yet been resolved in its full generality, partial results and related discussions on this topic can be found in the work of Shook \cite{shook}. Corollary \ref{Corol} provides some support for this conjecture.

\medskip
\noindent\textbf{Notation}. Throughout the paper, we use standard graph theory notation. Let $G=(V(G),E(G))$ be a graph. For any \(v\in V(G)\), we denote by \(N_{G}(v)\) the set of \emph{neighbors} of \(v\) in \(G\), by \(N_{G}[v]\) the set \(N_{G}(v)\cup\{v\}\), and by \(d_G(v)\) the \emph{degree} of \(v\) in \(G\). For any $S\subseteq V(G)$, let $G[S]$ denote the induced subgraph of $G$ on $S$, and $N(S)=\bigcup_{v\in S} N(v)$. We will drop the reference to $G$ when there is no danger of confusion.

\section{Proof of Conjecture \ref{1111}}
For convenience, we state Conjecture \ref{1111} as the following theorem.
\begin{thm}\label{thm:h_realizable_main}
Let \( n \in \mathbb{N} \) and \( d_1 \geq d_2 \geq \dots \geq d_n \ge h\) be positive integers.  The sequence \(\pi= (d_1, d_2, \dots, d_n) \) is \( H_h \)-realizable if and only if \( \sum_{i=1}^n d_i \) is even, \( n \) is a multiple of \( h+1 \), and for every \( k \in [n] \),
\begin{align}\label{2025}
\sum_{i=1}^k d_i \leq k(k-1) + \sum_{i=k+1}^{k+1+h-s} \min\{d_i - h + s, k\} + \sum_{i=k+1+h-s+1}^n \min\{d_i - h, k\},
\end{align}
where \( s \in \{0, \dots, h\} \) and \( s \equiv k \pmod{h + 1} \).
\end{thm}

\begin{proof}
To prove the necessity, let \( G \) be a realization of the sequence \( \pi=(d_1, d_2, \dots, d_n) \), where \( H \) is a spanning subgraph of \( G \). Since the sequence \( \pi=(d_1, d_2, \dots, d_n) \) is graphic, it follows that \( \sum_{i=1}^n d_i \) is even. Moreover, because \( H \) is a spanning subgraph of \( G \), it immediately follows that \( n \) is a multiple of \( h+1 \) and \( d_n \geq h \). For any \( k \in [n] \), we consider the graph \( G' \) obtained from \( G \) by deleting all edges in the intersection \( E\left(G\left[\{v_j \mid j > k\}\right]\right) \cap E(H) \).
Let \( i \in \left[\frac{n}{h+1}\right] \) be such that \( s \equiv k \pmod{h + 1} \) and \( v_k \in V_i \). Then, the following degree properties hold for vertices in \( G' \).
\begin{enumerate}
    \item For any vertex \( v \in \{v_j \mid v_j \in V_i \text{ and } j > k\} \), we have \( d_{G'}(v) = d_G(v) - h + s \).
    \item For any vertex \( v \in \bigcup_{j=i+1}^{\frac{n}{h+1}} V_j \), we have \( d_{G'}(v) = d_G(v) - h \).
\end{enumerate}
By applying the  Theorem \ref{thm:erdos_gallai} to the graph \( G' \), we can thereby deduce the inequality \eqref{2025}, completing the proof of necessity.

For sufficiency, let a subrealization of a non-increasing sequence \( (d_1, \dots, d_n) \) be defined as a graph $G$ with vertices \( v_1, \dots, v_n \) such that \( d(v_i) \leq d_i \) for all \( 1 \leq i \leq n \), where \( d(v_i) \) denotes the degree of vertex \( v_i \) in the graph. Given a sequence \( (d_1, \dots, d_n) \) with an even total degree that satisfies \eqref{2025}, we construct a realization of the sequence by successive subrealizations. The initial subrealization in this construction is precisely the graph \( H \).

For a subrealization, the critical index \( r \) is the largest index such that \( d(v_i) = d_i \) for all \( 1 \leq i < r \). While \( r \leq n \), we construct a new subrealization containing \( H \) as a spanning subgraph, with a smaller deficiency \( d_r - d(v_r) \) at \( v_r \), and without altering the degrees of vertices \( v_i \) for \( i < r \) (the degree list is lexicographically increasing). This process ends only when the subrealization is a realization of \( (d_1, d_2, \dots, d_n) \).

Let \( S = \{v_{r+1}, \dots, v_n\} \). We require that \( E(G[S]) = E(H[S]) \) for every subrealization, which certainly holds initially. We write \( v_i \leftrightarrow v_j \) if \( v_i v_j \in E(G) \); otherwise, we write \( v_i \not\leftrightarrow v_j \).

\textbf{Case 1.} \( v_r \not\leftrightarrow v_i \) for some vertex \( v_i \) such that \( d(v_i) < d_i \).
Add the edge \( v_r v_i \).

\textbf{Case 2.} \( v_r \not\leftrightarrow v_i \) for some \( i \) with \( i < r \).

Let \( v_i \in V_p \) and \( v_r \in V_s \). Since \( v_i \not\leftrightarrow v_r \), and given that every subrealization contains \( H \) as a spanning subgraph, we have \( p \neq s \). Because \( i < r \), it follows that \( p < s \). Let \( A = \{ v \mid v \in V_p \text{ and } v \notin N(v_r) \} \) and \( B = \{ v \mid v \in V_p \text{ and } v \in N(v_r) \} \). Let \( C = \{ v_j \mid v_j \in V_s \text{ and } j \leq r \} \) and \( D = \{ v_j \mid v_j \in V_s \text{ and } j > r \} \). Let \( F = N(v_i) \setminus N(v_r) \).

\begin{claim} \label{v_k v_r}
If \( d_r - d(v_r) = 1 \), then there exists some \( k > r \) such that \( d(v_k) < d_k \), \( v_k \in V_s \), and \( N(v_k) \subseteq N[v_r] \).
\end{claim}
\begin{proof}
Since \( \sum_{i=1}^n d_i - \sum_{i=1}^n d(v_i) \) is even, there exists an index \( k > r \) such that \( d(v_k) < d_k \). Case 1 applies unless \( v_r \leftrightarrow v_k \). Let \( v_r \in V_s \) and \( v_k \in V_t \).
If \( s \neq t \), then since \( i < r < k \), we have \( v_i \notin V_t \), and thus \( N(v_i) \cap V_t = \emptyset \). Otherwise, suppose \( v_j \in N(v_i) \cap V_t \). Since \( s \neq t \) and \( k > r \), it follows that \( t > s \), and hence \( j > r \). In this case, \( v_j v_i \notin E(H) \), so we replace the edge \( v_j v_i \) with the edge \( v_i v_r \).
Since \( N(v_i) \cap V_t = \emptyset \), we have \( |N(v_k) \setminus N[v_i]| \geq h \). Moreover, since \( d_i = d(v_i) \geq d_k > d(v_k) \), it follows that \( |N(v_i) \setminus N[v_k]| \geq h + 1 \). Because \( d_H(v_i) = h \), there exists a vertex \( v_{\ell} \in N(v_i) \setminus N[v_k] \) such that \( v_{\ell} v_i \notin E(H) \). Since \( v_r \leftrightarrow v_k \), \( v_{\ell} \) and \( v_r \) are distinct vertices.
To ensure that \( E(G[S]) = E(H[S]) \) for every subrealization, we proceed as follows, replacing \( \{v_{\ell} v_i\} \) with \( \{v_{\ell} v_k, v_i v_r\} \) if \( \ell < r \) and with \( \{v_i v_r\} \) if \( \ell > r \). Thus, \( v_r, v_k \in V_s \). If there exists a vertex \( u_3 \in N(v_k) \setminus N[v_r] \), then we replace the edge \( u_3 v_k \) with the edge \( u_3 v_r \).
\end{proof}

\begin{claim}\label{|A|2}
    \( F \neq \emptyset \), \( F \subseteq A \), and \( |A| \geq 2 \).
\end{claim}
\begin{proof}
    Since \( d(v_i) = d_i \geq d_r > d(v_r) \), we have \( |F| \geq 1 \). Suppose there exists \( v_{\ell} \in F \) such that \( v_{\ell} \notin A \). Then \( v_{\ell} v_i \notin E(H) \). If \( d_r - d(v_r) \geq 2 \), then we replace \( v_{\ell} v_i \) with the edge set \( \{v_{\ell} v_r, v_i v_r\} \). If \( d_r - d(v_r) = 1 \),  then there exists an index \( k > r \) such that \( d(v_k) < d_k \) by Claim \ref{v_k v_r}. If \( v_i \leftrightarrow v_k \), then we replace the edge \( v_i v_k \) with \( v_i v_r \) (since \( k > r > i \) and \( p < s \), \( v_i v_k \notin E(H) \)). If \( v_i \not\leftrightarrow v_k \), then we replace the edge \( v_i v_{\ell} \) with the edge set \( \{v_{\ell} v_r, v_i v_k\} \). This shows that \( F \subseteq A \). Since \( |F| \geq 1 \), \( F \subseteq A \), and \( v_i \in A \) by definition, we have \( |A| \geq 2 \).
\end{proof}

\begin{claim}    \label{A D}
\( N(A) \cap D = \emptyset \).
\end{claim}
\begin{proof}
Otherwise, there exists \( u \in N(A) \cap D \); without loss of generality, suppose \( v_i \in N(u) \cap A \); we replace \( v_i u \) with $v_iv_r$.
\end{proof}
\begin{claim} \label{AC}
\( N(A) \setminus (V_p \cup V_s) \subseteq N(v) \) for any \( v \in C \).
\end{claim}
\begin{proof}
Otherwise, without loss of generality, suppose that \( u \in N(A) \setminus (V_p \cup V_s) \), \( v_s \in A \), \( u \leftrightarrow v_s \), \( u \not\leftrightarrow v_j \), and \( v_j \in C \).
If \( v_j = v_r \) and \( d_r - d(v_r) \geq 2 \), the we replace \( uv_s \) with the edge set \( \{v_s v_r, uv_r\} \).
If \( v_j = v_r \) and \( d_r - d(v_r) = 1 \), then by Claim \ref{v_k v_r}, there exists some \( k > r \) such that \( d(v_k) < d_k \). Since \( N(v_k) \subseteq N[v_r] \), it follows that \( v_k \not\leftrightarrow v_s \). We then replace the edge \( v_s u \) with the edge set \( \{u v_r, v_s v_k\} \).
If \( v_j \neq v_r \) and \( d_r - d(v_r) \geq 2 \), then we have \( d(v_j) = d_j \geq d_r > d(v_r) + 1 \) as \( j < r \). Thus, there exists a vertex \( x \) distinct from \( v_s \) such that \( x \in N(v_j) \setminus N[v_r] \). We then replace the edge set \( \{uv_s, v_j x\} \) with the edge set \( \{uv_j, v_s v_r, x v_r\} \).
If \( v_j \neq v_r \) and \( d_r - d(v_r) = 1 \), then \( d(v_j) = d_j \geq d_r > d(v_r) \) as \( j < r \). Thus, there exists a vertex \( x \) such that \( x \in N(v_j) \setminus N[v_r] \) (possibly \( x = v_s \)). By Claim \ref{v_k v_r}, \( v_s \not\leftrightarrow v_k \). We replace the edge set \( \{uv_s, v_j x\} \) with the edge set \( \{uv_j, x v_r, v_s v_k\} \).
\end{proof}

\begin{claim}\label{c-1}
    For any \( v \in C \), we have \( |N(v) \cap V_p| \leq |C| - 1 \) and \( |N(v_r) \cap V_p| = |B| \leq |C| - 2 \).
\end{claim}
\begin{proof}
For any \( v_j \neq v_r \) with \( v_j \in C \), by Claim \ref{A D}, we have \( |N(v_i) \cap V_s| \leq |C| \). Moreover, since \( v_i \not\leftrightarrow v_r \) and \( v_r \in C \), it follows that \( |N(v_i) \cap V_s| \leq |C|-1 \). By Claim \ref{AC}, we can obtain \( |N(v_i) \setminus (V_p \cup V_s)| \leq |N(v_j) \setminus (V_p \cup V_s)| \). If \( |N(v_j) \cap V_p| \geq |C| \), then
\begin{align*}
d(v_j) &= |N(v_j) \cap V_p| + |N(v_j) \cap V_s| + |N(v_j) \setminus (V_p \cup V_s)| \\
&\geq |C| + h + |N(v_j) \setminus (V_p \cup V_s)|.
\end{align*}
However,
\begin{align*}
d(v_i) &= |N(v_i) \cap V_p| + |N(v_i) \cap V_s| + |N(v_i) \setminus (V_p \cup V_s)| \\
&\leq h + |C| - 1 + |N(v_j) \setminus (V_p \cup V_s)|.
\end{align*}
This implies that \( d(v_i) < d(v_j) \), a contradiction. Thus, we complete the proof that \( |N(v) \cap V_p| \leq |C| - 1 \) for any \( v \in C\setminus\{v_r\} \).

Note that \( |N(v_r) \cap V_p| = |B| \) by the definition of \( B \). By Claim \ref{A D}, we have \( D \subseteq N(v_r) \setminus N(v_i) \). Thus, \( |N(v_r) \setminus N(v_i)| \ge |D| = h + 1 - |C| \). Since \( d(v_i) = d_i \ge d_r > d(v_r) \), it follows that \( |N(v_i) \setminus N(v_r)| = |F| \ge h + 2 - |C| \). By Claim \ref{|A|2}, we have \( F = N(v_i) \setminus N(v_r) \subsetneq A \). Moreover, since \( v_i \notin F \) but \( v_i \in A \), we get \( |A| \ge h + 3 - |C| \). Given that \( |B| = h + 1 - |A| \), we therefore obtain \( |B| \le |C| - 2 \).
\end{proof}
\begin{claim} \label{3}
$|C|\geq 3$.
\end{claim}
\begin{proof}
    Suppose that \( |C| \leq 2 \).
If \( |C| = 1 \), then \( N(v_i) \cap C = \emptyset \). By Claim \ref{A D}, \( N(v_i) \cap D = \emptyset \), so \( N(v_i) \cap V_s = \emptyset \), and thus \( |N(v_r) \setminus N(v_i)| \geq h \). Since \( d_i = d(v_i) \geq d_r > d(v_r) \), we have \( |N(v_i) \setminus N(v_r)| \geq h + 1 \). This implies that there exists \( u \in N(v_i) \setminus N(v_r) \) such that \( u \notin A \), which contradicts Claim \ref{|A|2}.
Suppose that\( |C| = 2 \), and let \( C = \{v_j, v_r\} \). By Claim \ref{c-1}, we have \( |N(v_j) \cap V_p| \leq |C| - 1 = 1 \). Note that by Claim \ref{|A|2}, we have \( |A| \geq 2 \), which implies that there exists a vertex \( v_\ell \in A \) such that \( N(v_\ell) \cap C = \emptyset \). By Claim \ref{A D}, \( N(v_\ell) \cap D = \emptyset \). Thus, \( |N(v_r) \setminus N(v_\ell)| \geq h \). Since \( d_\ell = d(v_\ell) \geq d_r > d(v_r) \), we have \( |N(v_\ell) \setminus N(v_r)| \geq h + 1 \). This implies that there exists \( u \in N(v_\ell) \setminus N(v_r) \) such that \( u \notin V_p \cup V_s \), which contradicts Claim \ref{AC}.
\end{proof}

\begin{claim}\label{88}
   There do not exist vertices \( v_\ell \in V_p \) and \( v_m \in C \) such that \( |N(v_\ell) \cap C| < |N(v_m) \cap V_p| \).
\end{claim}
\begin{proof}
   Suppose that there exist vertices \( v_\ell \in V_p \) and \( v_m \in C \) such that \( |N(v_\ell) \cap C| < |N(v_m) \cap V_p| \).
If \( v_\ell \in A \), by Claim \ref{AC}, we can obtain \( |N(v_\ell) \setminus (V_p \cup V_s)| \leq |N(v_m) \setminus (V_p \cup V_s)| \). By Claim \ref{A D}, \( N(v_\ell) \cap D = \emptyset \). This implies that
\begin{align*}
d(v_\ell) &= |N(v_\ell) \cap V_p| + |N(v_\ell) \cap V_s| + |N(v_\ell) \setminus (V_p \cup V_s)| \\
&= |N(v_\ell) \cap V_p| + |N(v_\ell) \cap C| + |N(v_\ell) \setminus (V_p \cup V_s)| \\
&< |N(v_m) \cap V_p| + |N(v_m) \cap V_s| + |N(v_m) \setminus (V_p \cup V_s)| \\
&= d(v_m),
\end{align*}
which contradicts the fact that \( d(v_\ell) \geq d(v_m) \).
Thus, we have \( v_\ell \in B \), so \( v_\ell \leftrightarrow v_r \).
By Claim \ref{c-1} and \( |N(v_\ell) \cap C| < |N(v_m) \cap V_p| \), we can conclude that
\begin{align}\label{22}
|N(v_\ell) \cap C| \leq |C| - 2.
\end{align}
By \eqref{22} and Claim \ref{3}, there exists \( v_t \in C \) other than \( v_r, v_m \) such that \( v_t \not\leftrightarrow v_\ell \) as \( v_\ell \leftrightarrow v_r \).
Note that
\begin{align*}
d(v_l) &= |N(v_\ell) \cap V_p| + |N(v_\ell) \cap C| + |N(v_\ell) \cap D| + |N(v_\ell) \cap (V(G) \setminus (V_p \cup V_s))|,
\end{align*}
and
\begin{align*}
d(v_m) &= |N(v_m) \cap V_p| + |N(v_m) \cap V_s| + |N(v_m) \cap (V(G) \setminus (V_p \cup V_s))|.
\end{align*}
Since \( |N(v_\ell) \cap C| < |N(v_m) \cap V_p| \), \( |N(v_\ell) \cap V_p| = h \), \( |N(v_m) \cap V_s| = h \), and \( d(v_\ell) = d_\ell \geq d_m \geq d(v_m) \), there either exists a vertex \( u \in N(v_\ell) \setminus N(v_m) \) such that \( u \notin V_p \cup V_s \), or there exists a vertex \( u \in N(v_\ell) \cap D \). If \( d_r - d(v_r) \geq 2 \), it follows that \( d(v_t) = d_t \geq d_r > d(v_r) + 1 \), so there exists \( v \in N(v_t) \setminus N[v_r] \) with \( u \neq v \). If \( d_r - d(v_r) = 1 \), by Claim \ref{v_k v_r}, we have \( v \not\leftrightarrow v_k \) (possibly \( u = v \)).

If \( m = r \) and there exists a vertex \( u \in N(v_\ell) \setminus N(v_m) \) such that \( u \notin V_p \cup V_s \), then we proceed by implementing edge replacements tailored to the value of \( d_r - d(v_r) \); for the case \( d_r - d(v_r) \geq 2 \), we replace the edge set \( \{v_\ell u, v_t v\} \) with the edge set \( \{v_\ell v_t, u v_r, v v_r\} \), and for the case \( d_r - d(v_r) = 1 \), by Claim \ref{v_k v_r}, there exists an index \( k > r \) such that \( d(v_k) < d_k \); if \( v \notin S \), we replace the edge set \( \{v_\ell u, v_t v\} \) with the edge set \( \{v_t v_\ell, u v_r, v v_k\} \), and if \( v \in S \), we replace the edge set \( \{v_\ell u, v_t v\} \) with the edge set \( \{v_t v_\ell, u v_r\} \).
If \( m = r \) and there exists a vertex \( u \in N(v_\ell) \cap D \), then we replace the edge set \( \{v_\ell u, v_t v\} \) with the edge set \( \{v_t v_\ell, v v_r\} \).

If \( m \neq r \) and there exists a vertex \( u \in N(v_\ell) \setminus N(v_m) \) such that \( u \notin V_p \cup V_s \), if \( d_r - d(v_r) \geq 2 \), it follows that \( d(v_m) = d_m \geq d_r > d(v_r) + 1 \), so there exists a vertex \( w \) distinct from \( v \) such that \( w \in N(v_m) \setminus N[v_r] \), and we replace the edge set \( \{v_\ell u, v_t v, v_m w\} \) with the edge set \( \{v_\ell v_t, u v_m, v v_r, w v_r\} \).
If \( d_r - d(v_r) = 1 \) (possibly \( v = w \)), then, by Claim \ref{v_k v_r}, there exists an index \( k > r \) such that \( d(v_k) < d_k \) and \( v \not\leftrightarrow v_k \); if \( v \notin S \), then we replace the edge set \( \{v_\ell u, v_t v, v_m w\} \) with the edge set \( \{v_t v_\ell, u v_m, w v_r, v v_k\} \), and if \( v \in S \), then we replace the edge set \( \{v_\ell u, v_t v, v_m w\} \) with the edge set \( \{v_t v_\ell, u v_m, w v_r\} \).
If \( m \neq r \) and there exists a vertex \( u \in N(v_\ell) \cap D \), then we replace the edge set \( \{v_\ell u, v_t v\} \) with the edge set \( \{v_t v_\ell, v v_r\} \).
\end{proof}

\begin{claim}\label{rh+1}
    \(|C| = h + 1\), \(0 \equiv r \pmod{h+1}\), and \(d_r - d(v_r) \geq 2\).
\end{claim}
\begin{proof}
    If \( |C| \leq h \), then by Claim \ref{88}, for any \( v_\ell \in V_p \) and \( v_m \in C \), \( |N(v_\ell) \cap C| \geq |N(v_m) \cap V_p| \); noting that \( |V_p| = h + 1 \), this implies that unless \( e(V_p, C) = 0 \), we have \( \sum_{v \in V_p} |N(v) \cap C| > \sum_{v \in C} |N(v) \cap V_p| \), a contradiction.
If \( e(V_p, C) = 0 \), then \( N(v_i) \cap C = \emptyset \). By Claim \ref{A D}, \( N(v_i) \cap D = \emptyset \), so \( N(v_i) \cap V_s = \emptyset \), and thus \( |N(v_r) \setminus N(v_i)| \geq h \). Since \( d_i = d(v_i) \geq d_r > d(v_r) \), we have \( |N(v_i) \setminus N(v_r)| \geq h + 1 \). This implies there exists \( u \in N(v_i) \setminus N(v_r) \) such that \( u \notin A \), which contradicts Claim \ref{|A|2}.

By the definition of \( C \), \( |C| = h + 1 \) implies that \( 0 \equiv r \pmod{h+1} \). If \( d_r - d(v_r) = 1 \), by Claim \ref{v_k v_r}, there exists an index \( k > r \) such that \( v_k, v_r \in V_s \); then \( |D| \geq 1 \), which contradicts \( |C| = h + 1 \).
\end{proof}

\begin{claim}\label{9999}
    For any vertex \( v_j \in V_p \), we have \( |N(v_j) \cap V_s| \leq h \).
\end{claim}
\begin{proof}
    If there exists a vertex \( v_j \in V_p \) such that \( |N(v_j) \cap V_s| = h + 1 \), then by Claim \ref{88}, for any \( u \in V_p \) and \( v \in V_s \), \( |N(u) \cap V_s| \geq |N(v) \cap V_p| \). Since \( |N(v_j) \cap V_s| = h + 1 \), and by Claims \ref{c-1} and \ref{rh+1} (which state that for any \( v \in C \), \( |N(v) \cap V_p| \leq |C| - 1 \) and \( C = V_s \)), it follows that \( \sum_{u \in V_p} |N(u) \cap V_s| > \sum_{u \in V_s} |N(u) \cap V_p| \), a contradiction.
\end{proof}

By Claim \ref{rh+1}, we have \( d_r - d(v_r) \geq 2 \). Moreover, since for any \( u \in V_p \) and \( v \in V_s \), \( d(u) \geq d(v) \), it follows that \( \sum_{v \in V_p} d(v) - \sum_{v \in V_s} d(v) \geq 2 \). Thus, there exists a vertex \( u \notin (V_p \cup V_s) \) such that \( |N(u) \cap V_p| > |N(u) \cap V_s| \).
Since \( |N(u) \cap V_p| > |N(u) \cap V_s| \), we conclude that \( u \notin N(A) \) by Claim \ref{AC}. Let \( v_\ell \in V_p \) satisfy \( v_\ell \leftrightarrow u \); then \( v_\ell \in B \). By Claim \ref{9999} and since \( v_\ell \in B \), there exists a vertex \( v_t \in V_s \) distinct from \( v_r \) such that \( v_\ell \not\leftrightarrow v_t \).
Since \( u \notin N(A) \), \( |N(u) \cap V_p| > |N(u) \cap V_s| \), and  \( |A| \geq 2 \) by Claim \ref{|A|2}, there exists a vertex \( v_m \in V_s \) distinct from \( v_r \) and \( v_t \) such that \( v_m \not\leftrightarrow u \). By Claim \ref{rh+1}, we have \( d(v_t) = d_t \geq d_r > d(v_r) + 1 \) and \( d(v_m) = d_m \geq d_r > d(v_r) + 1 \). This implies that there exist two distinct vertices \( v, w \) such that \( v \in N(v_t) \setminus N[v_r] \) and \( w \in N(v_m) \setminus N[v_r] \) (possibly \( u = v \)). We replace the edge set \( \{v_\ell u, v_t v, v_m w\} \) with the edge set \( \{v_\ell v_t, u v_m, v v_r, w v_r\} \).

\textbf{Case 3.} Let \( s \equiv r \pmod{h+1} \), where \( v_1, \ldots, v_{r-1} \in N(v_r) \); and either \( |N(v_k) \cap \{v_1, \ldots, v_r\}| \neq \min\{r, d_k - h + s\} \) for some \( k \) with \( r < k \leq r + 1 + h - s \), or \( |N(v_k) \cap \{v_1, \ldots, v_r\}| \neq \min\{r, d_k - h\} \) for some \( k \) with \( k \geq r + 1 + h - s + 1 \).

For \( k > r \), \( |N(v_k) \cap \{v_1, \ldots, v_r\}| \leq r \) and \( d(v_k) \leq d_k \). Since \( E(G[S]) = E(H[S]) \) for every subrealization, we have the following bounds: \( |N(v_k) \cap \{v_1, \ldots, v_r\}| \leq d_k - h + s \) for indices \( k \) with \( r < k \leq r + 1 + h - s \), and \( |N(v_k) \cap \{v_1, \ldots, v_r\}| \leq d_k - h \) for indices \( k \) with \( k \geq r + 1 + h - s + 1 \). We therefore conclude that either \( |N(v_k) \cap \{v_1, \ldots, v_r\}| < \min\{r, d_k - h + s\} \) for some \( k \) with \( r < k \leq r + 1 + h - s \), or \( |N(v_k) \cap \{v_1, \ldots, v_r\}| < \min\{r, d_k - h\} \) for some \( k \) with \( k \geq r + 1 + h - s + 1 \).

Since \( d(v_k) < r \), Case 1 applies unless \( v_k \leftrightarrow v_r \) and there exists an index \( i < r \) such that \( v_k \not\leftrightarrow v_i \). In view of \( d(v_i) = d_i \geq d_r > d(v_r) \), there exists a vertex \( v_j \in N(v_i) \setminus N[v_r] \). Since \( v_1, \ldots, v_{r-1} \in N(v_r) \), it follows that \( j > r \), and thus \( v_j v_r \notin E(H) \). We replace the edge \( v_j v_i \) with the edge set \( \{v_j v_r, v_i v_k\} \).

\textbf{Case 4.} \( v_1, \ldots, v_{r-1} \in N(v_r), \text{ and } v_i \not \leftrightarrow v_j \text{ for some } i \text{ and } j \text{ with } i < j < r \).

Since \( d(v_i) \geq d(v_j) > d(v_r) \), there exist vertices \( v_l \in N(v_i) \setminus N[v_r] \) and \( v_m \in N(v_j) \setminus N[v_r] \) (possibly \( v_l = v_m \)). Since \( v_1, \ldots, v_{r-1} \in N(v_r) \), it follows that \( l > r \) and \( m > r \), and thus \( v_l v_i \notin E(H) \) and \( v_m v_j \notin E(H) \). We replace the edge set \( \{v_l v_i, v_m v_j\} \) with the edge set \( \{v_i v_j, v_l v_r\} \).

If none of these cases applies, then \( v_1, \ldots, v_r \) are pairwise adjacent, \( |N(v_k) \cap \{v_1, \ldots, v_r\}| = \min\{r, d_k - h + s\} \) for indices \( k \) with \( r < k \leq r + 1 + h - s \), and \( |N(v_k) \cap \{v_1, \ldots, v_r\}| = \min\{r, d_k - h\} \) for indices \( k \) with \( k \geq r + 1 + h - s + 1 \). Furthermore, since \( E(G[S]) = E(H[S]) \) for every subrealization,
\[
\sum_{i=1}^r d(v_i) = r(r-1) + \sum_{i=r+1}^{r+1+h-s} \min\{d_i - h + s, r\} + \sum_{i=r+1+h-s+1}^n \min\{d_i - h, r\},
\]
where \( s \in \{0, \ldots, h\} \) and \( s \equiv k \pmod{h + 1} \). By \eqref{2025}, \( \sum_{i=1}^r d_i \) is bounded by the right side. Hence we have already eliminated the deficiency at vertex $v_r$. Increase $r$ by 1 and continue.
\end{proof}

\end{document}